\documentclass[12pt]{amsart}
\usepackage{amsmath}
\usepackage{amssymb}

\usepackage{amsthm}
\usepackage[dvipsnames]{pstricks}
\usepackage{graphicx}

\numberwithin{equation}{section}

\newtheorem{theorem}[equation]{Theorem}
\newtheorem{lemma}[equation]{Lemma}
\newtheorem{proposition}[equation]{Proposition}
\newtheorem{corollary}[equation]{Corollary}

\newtheorem{conjecture}[equation]{Conjecture}

{\theoremstyle{definition}}
{\theoremstyle{definition}}
{\theoremstyle{definition}\newtheorem{remark}[equation]{Remark}}
{\theoremstyle{definition}\newtheorem{defn}[equation]{Definition}}

\def\N{{\mathbb N}}
\def\Z{{\mathbb Z}}

\def\K{{\mathbb K}}

\def\d{\partial}

\def\epsilon{\varepsilon}
\def\phi{\varphi}
\def\leq{\leqslant}
\def\geq{\geqslant}

\def\dim{{\rm dim}\,}

\def\deg{\hbox{\tt \rm deg}\,}

\def\N{{\mathbb N}}
\def\Z{{\mathbb Z}}

\def\phi{\varphi}

\def\dim{\text{\rm dim}\,}

\def\epsilon{\varepsilon}
\def\phi{\varphi}
\newcommand{\U}{{U}}

\newcommand{\mb}{\mathbb}
\newcommand{\CC}{\mb K} % this seemed easiest
\newcommand{\ZZ}{\mb Z}
\newcommand{\NN}{\mb N}
\newcommand{\KK}{\K}

\newcommand{\mc}{\mathcal}
\newcommand{\mf}{\mathfrak}

\newcommand{\DMO}{\DeclareMathOperator}
\DMO{\Ann}{Ann}
\newcommand{\V}{V\!ir}
\DMO{\GK}{GKdim}
\DMO{\gr}{gr}

\newcommand{\beq}{\begin{equation}}
\newcommand{\eeq}{\end{equation}}

\newcommand{\blank}{\mbox{$\underline{\makebox[10pt]{}}$}}

\newcommand{\del}{\partial}

\title{
Enveloping  algebras with just infinite Gelfand-Kirillov dimension}

\author{Natalia K. Iyudu, Susan J. Sierra}

\date{\today}

\address{School of Mathematics, The University of Edinburgh, Edinburgh EH9 3FD, United Kingdom}
\email{n.iyudu@ed.ac.uk}
\email{s.sierra@ed.ac.uk}

\keywords{Witt algebra, positive Witt algebra, Virasoro algebra,  Gelfand-Kirillov dimension}
\subjclass[2010]{Primary:  16S30, 17B68, 16P90; Secondary 17B65}

\begin{document}
\begin{abstract}
Let $\mf g$ be the Witt algebra or the positive Witt algebra.
It is well known that the enveloping algebra $U(\mf g )$ has intermediate growth and thus infinite Gelfand-Kirillov (GK-) dimension.
We prove that the GK-dimension  of  $U(\mf g)$ is {\em just infinite} in the sense that
any proper quotient of $U(\mf g)$ has polynomial growth.
 This proves a conjecture of Petukhov and the second named author for the positive Witt algebra.
 We also establish the corresponding results for quotients of the symmetric algebra $S(\mf g)$ by proper Poisson ideals.

In fact, we prove more generally that any central quotient of the universal enveloping algebra of the Virasoro algebra has just infinite GK-dimension.
We give several applications.
In particular, we easily compute the  annihilators of Verma modules over the Virasoro algebra.
\end{abstract}

\maketitle

%\tableofcontents

\section{Introduction}

Let ${\mathbb K}$ be a  field of characteristic zero,
and let $W$ be the {\em Witt algebra}, which has
${\mathbb K}$-basis
$$
\{e_n:n\in{\mathbb Z}\},
$$
with the Lie bracket
$$
[e_i,e_j]=(j-i)e_{i+j}.
$$
We let $W_+$, the {\em positive Witt algebra}, be the Lie subalgebra of $W$ spanned by $\{e_n: n \geq 1\}$.

The Witt algebra is a central quotient of the {\em Virasoro algebra}, $\V$, which has $\mb K$-basis
\[ \{e_n : n \in \ZZ\}\cup \{c\},\]
and Lie bracket
\[ [ e_i, e_j ] = (j-i)e_{i+j} + \frac{i^3-i}{12} \delta_{i+j,0} c, \quad \mbox{$c$ central}.\]
The Virasoro algebra and its representations are ubiquitous in conformal field theory.

These algebras were  testing examples for the  fundamental and important question of whether there is an infinite-dimensional Lie algebra with a (left and right) noetherian enveloping algebra.
This question has been asked by many people, including Ralph Amayo and Ian Stewart \cite[Question 27, p. 396]{AS}, Ken Brown \cite[Question B]{Brown}, Jacques Dixmier, and Victor Latyshev, and  conjecturally  has a negative answer; see \cite[Conjecture~0.1]{SW1}.
Recently it was shown by the second named author and Chelsea Walton \cite{SW1} that  the conjecture holds for the Lie algebras above:  that is, $U(W_+)$, $U(W)$, and $U(\V)$ are not left or right Noetherian.
The question is still unsolved in full generality.

However, the two-sided ideal structure of these enveloping algebras is extremely sparse, and it seems possible that they satisfy the ascending chain condition for two-sided ideals, a property sometimes known as being {\em weakly Noetherian}.
Further, two-sided ideals of $W$ and $W_+$ are known to be large, and Petukhov and the second named author have conjectured:
\begin{conjecture}\label{conj1}
{\rm (\cite[Conjecture~1.2]{PS})} The universal enveloping algebra $U(W_+)$ has just infinite GK-dimension; that is, if $I$ is a non-zero ideal of $U(W_+)$,
then $U(W_+)/I$ has polynomial growth.
\end{conjecture}
\noindent (In this paper, we use {\em polynomial growth} as a synonym for finite GK-dimension.)

This conjecture was proved in \cite{PS} in the particular case that the ideal $I$ is generated by quadratic expressions in the $e_i$.
In this paper we establish the conjecture in full, and generalise our  arguments to prove that $\U(W)$ and indeed any central factor of $\U(\V)$ has just infinite GK-dimension in the sense above.
Our main result is:
\begin{theorem}\label{ithm:U}
{\rm (See Theorems~\ref{theoremW+}, \ref{theoremW} and  \ref{theoremWkappa}.) }
The algebras $U(W_+)$, $U(W)$, and $U(\V)/(c-\kappa)$, for any $\kappa \in \K$, have just infinite GK-dimension.
In particular, Conjecture~\ref{conj1} holds.
\end{theorem}

We note that these  algebras are all finitely generated  and as  many authors have noticed \cite{Smith, Ufn, KKM}
have  intermediate  growth (and thus infinite GK-dimension).
The  natural set of normal words is monomials of the form $e_{i_1},\dots,e_{i_k}, \quad i_1\leq i_2 \dots \leq i_k$.
Here the number of monomials with $i_1 + i_2 + \dots + i_k = n$ is in fact  the number of partitions of $n$, which was shown by Hardy to be bounded by  $e^{c \sqrt n}$, an estimate later dramatically improved by Ramanujan.

We also consider the induced Poisson structures on the  symmetric algebras $S(W_+)$, $S(W)$, and $S(\V)$, and prove:
\begin{theorem}\label{ithm:S}
{\rm (See Theorems~\ref{theoremP} and \ref{theoremPkappa}.) }
Let $I$ be a proper Poisson ideal of $S(W_+)$.
Then $S(W)/I$ has polynomial growth.
Similar statements hold for $S(W)$ and for $S(\V)/(c-\kappa)$, for any $\kappa \in \K$.
\end{theorem}

The Lie algebra $\V$ has a triangular decomposition, and so the classical notion of a {\em Verma module} makes sense.  
As an application of Theorem~\ref{ithm:U}, we 
easily compute annihilators of Verma modules for $\V$:  
\begin{theorem}\label{ithm:Verma}
{\rm (see Corollary~\ref{cor:Verma})}
Let $M$ be a Verma module over $\V$ with central charge $\kappa$.
Then the annihilator of $M$ in $\U(\V)$ is the ideal $(c-\kappa)$.
As a result, for any $\kappa \in \mathbb K$ the algebra $\U(\V)/(c-\kappa)$ is primitive.
\end{theorem}

%give an easy proof  that annihilators of Verma modules over $\V$ are generated by some $c-\kappa$ (see Corollary~\ref{cor:Verma}); 
Theorem~\ref{ithm:Verma} is an unpublished result of Wallach \cite{Wallach}.

We also obtain:
\begin{proposition}\label{iprop:ACC}
 {\rm (See Proposition~\ref{prop:ACC}.)}
The algebras  $\U(W)$, $\U(W_+)$, and $\U(\V)$ all satisfy the ascending chain condition for  completely prime ideals.
\end{proposition}

\begin{proposition}\label{iprop:Hopfian}
{\rm (See Proposition~\ref{prop:HB}.)}
Let $U(\mc W)$ be either $U(W_+)$, $U(\V)$, or $U(\V)/(c-\kappa)$ for some $\kappa \in \KK$.
Then any epimorphism of $U(\mc W)$ is an isomorphism.
\end{proposition}

This answers a question of Rowen and Small \cite[Section~4]{RS}.

%We do not know of a prime ideal of any of these algebras which is not completely prime.

The Witt algebra $W$ is a simple, graded, Lie algebra of polynomial growth.
Such algebras were famously classified by Mathieu \cite{Mathieu}.
It is interesting to ask which of these Lie algebras have enveloping algebras with just infinite GK-dimension.
This is the subject of  ongoing research.

\vspace{1em}

\noindent {\bf Methods:}
As notation, we will write the symmetric algebra $S(W_+) = \K[x_1, x_2, \dots]$, where $x_i$ corresponds to $e_i \in W_+$.
Likewise,
\[S(W) = \K[\dots, x_{-1}, x_0, x_1, \dots] \quad \mbox{and} \quad S(\V) = \KK[\dots, x_{-1}, x_0, x_1, \dots, c].\]

The main idea in the proofs of both Theorems~\ref{ithm:U} and \ref{ithm:S} is to show that if $g$ is a nonzero element of $S(W_+)$ or $S(W)$, then for `almost all' monomials $m$ in the $x_i$ (more precisely, for long enough monomials on big enough letters), the Poisson ideal generated by $g$ contains an element with leading term $m$. (See Lemma~\ref{lemmaP1}.)
In the  same way, we show that if $g$ is a nonzero element of $U(W_+)$ or $U(W)$, then for almost all monomials $m$ in the variables $e_i$, the two-sided ideal  generated by $g$ contains an element with leading term $m$. (See Lemmata~\ref{lem:newreduction} and \ref{lem:NS}.)

We summarise the argument for $S(W_+)$.  Let $I$ be the Poisson ideal generated by $0 \neq g \in S(W_+)$.
We introduce the following ordering on monomials in the $x_i$.
Denote by  $\deg(m)=i_1+{\dots}+i_k$  the {\it degree} of a monomial $m=x_{i_1}\dots x_{i_k}$, and
let the  {\it length} of $m$ be
$|m|=|x_{i_1}\dots x_{i_k}|=k$.

Then the ordering on the set  of
commutative monomials in the $x_i$ is defined as follows. For two monomials $m_1$ and $m_2$, we write $m_1<m_2$ if
\begin{itemize}
\item $|m_1|<|m_2|$ or
\item $|m_1|=|m_2|$ and $\deg(m_1)<\deg(m_2)$ or
\item $|m_1|=|m_2|$, $\deg(m_1)=\deg(m_2)$ and $m_1$ is less than $m_2$ with respect to the left-to-right
lexicographical order when both $m_1$ and $m_2$ are written in increasing order: $m_1=x_{i_1}\dots x_{i_k}$,
$m_2=x_{j_1}\dots x_{j_k}$ with $i_1\leq i_2\leq\dots\leq i_k$, $j_1\leq j_2\leq\dots\leq j_k$.
\end{itemize}

We show in Lemma~\ref{lemmaP1}
that all sufficiently long  monomials on sufficiently `big' letters
can be written,  modulo $I$, as a sum of smaller  monomials.
Here a
 letter is called {\it big} if it is bigger than
$n= \{ \max (2i+1) \, | \, x_i$ occurs in $g$\}.
By this means we are able to introduce  a `normal form' for monomials from $S(W_+)$. Namely, any element of $S(W_+)$ can be written, modulo $I$,  as a linear combination  of  monomials
of the form $m = uv$,
where $u$ is a monomial on the finite set of letters $ x_1,\dots x_{n-1}$, and $ v$  is a monomial  of restricted length on the set of big letters.
  A similar normal form works for $U(W_+)$.

In the case of $W$ (more generally, central quotients of $\V$) the normal form is slightly different.
Let $0 \neq J \triangleleft U(W)$.  Any element of $U(W)$ can be written, modulo $J$, as a linear combination of monomials $m = u_1vu_2$,
where $v$ is a monomial on the finite set of letters $ e_{1-n},\dots e_{n-1}$,  $ u_1$ is a monomial of bounded length on letters smaller than $e_{-n}$, and $u_2$ is a monomial of bounded length on letters larger than $e_n$.

Counting  the growth of these normal monomials  gives us a polynomial estimate which bounds the growth of the quotient algebra.
In the case of the full Witt algebra the  growth counting is somewhat more involved, since the usual degree function ${\rm deg }(u)=i_1+\dots +i_k$ will not supply us  a finite filtration on $U(W)/J$.
See Section~\ref{filt} for the details of how this issue is resolved.

\vspace{1em}

\noindent {\bf Notation:}
Throughout we fix the following notation.  We denote the set of non-negative integers by $\NN$.
If $R$ is a ring and $g \in R$, the two-sided ideal generated by $g$ is denoted $(g)$.
If $R$ is a Poisson algebra, the Poisson ideal generated by $g$ is denoted $\{(g)\}$.

Let $S({\mathcal W})$ denote either $S(W_+)$ , $S(\V)$, or   $S(\V)/(c-\kappa)$ for some $\kappa \in \KK$.
Likewise, let  $U(\mc W)$ be either $U(W_+)$, $U(\V)$,  or $U(\V)/(c-\kappa)$.
Our convention is that $e_1, e_2, \dots$ denote elements of $U(\mathcal{W})$, and that $x_1, x_2, \dots$ are the corresponding elements of $S(\mathcal{W})$.
Monomials  $x_{i_1}x_{i_2}\dots x_{i_k} \in S(\mc W)$  are usually  written in increasing order (if not specified otherwise), that is  $i_1\leq i_2 \leq \dots  \leq i_k$. The same shape  $e_{i_1}e_{i_2}\dots e_{i_k} $, where $i_1\leq i_2 \leq  \dots  \leq i_k$,  of a monomial is considered normal in $U({\mathcal W})$, and such a monomial is called  {\em standard}.

\section{The symmetric algebra of the positive Witt algebra}\label{P}

Our main goal in this section is to prove the following theorem, which gives the first statement of Theorem~\ref{ithm:S}:

\begin{theorem}\label{theoremP}
Let $I$ be a nonzero Poisson ideal of $S(W_+)$.
 Then $A=S(W_+)/I$ has
polynomial growth.
\end{theorem}

Let $X=\{x_1,x_2,\dots\}$, so $S(W_+) = \K[X]$, with Poisson bracket induced by defining $\{x_i, x_j\} = (j-i) x_{i+j}$.
We equip $S(W_+)$ with the
degree grading given by setting  $\deg(x_j)=j$.
Thus the {\it degree} of a monomial $m=x_{i_1}\dots x_{i_k}$ is
$\deg(m)=i_1+{\dots}+i_k$. Of course, we  have the natural grading by the {\it length} of monomials as well,  which we write as
$|x_{i_1}\dots x_{i_k}|=k$.

The key technique in the proof  is the following {\em reduction formula} for elements of $A = S(W_+)/I$ in the case that $I=\{(g)\}$ is the Poisson ideal generated by a single, nonzero, degree-homogeneous polynomial $g$.

\begin{lemma}\label{lemmaP1} Let $I=\{(g)\}$ be the Poisson ideal in $\K[X]=S(W_+)$ generated by a nonzero degree-homogeneous polynomial $g\in\K[X]$.
There exist positive integers $k$ and $n$ such that every
monomial $m=x_{j_1}\dots x_{j_k}$ such that $j_\ell\geq n$ for $1\leq \ell \leq k$ satisfies
\beq\label{RF}
m=h+\sum c_s m_s,
\eeq
where $h\in I$ is degree-homogeneous with $\deg(h) = \deg(m)$, the sum is finite, $c_s\in\K$, and the $m_s$ are monomials of degree $\deg(m)$ such that for each $s$, either $|m_s|<k$ or
$|m_s|=k$ and $i<n$ for at least one of the letters $x_i$ featuring in $m_s$.
\end{lemma}

To prove Lemma~\ref{lemmaP1}, we introduce the following two orderings on the set $[X]$ of
commutative monomials in $X$. For two monomials $m_1$ and $m_2$, we write $m_1<m_2$ if
\begin{itemize}
\item $|m_1|<|m_2|$ or
\item $|m_1|=|m_2|$ and $\deg(m_1)<\deg(m_2)$ or
\item $|m_1|=|m_2|$, $\deg(m_1)=\deg(m_2)$ and $m_1$ is less than $m_2$ with respect to the left-to-right
lexicographic order when both $m_1$ and $m_2$ are written in {\em increasing} order: $m_1=x_{i_1}\dots x_{i_k}$,
$m_2=x_{j_1}\dots x_{j_k}$ with $i_1\leq i_2\leq\dots\leq i_k$, $j_1\leq j_2\leq\dots\leq j_k$.
\end{itemize}
Similarly, we write $m_1\prec m_2$  if
\begin{itemize}
\item $|m_1|<|m_2|$ or
\item $|m_1|=|m_2|$ and $\deg(m_1)<\deg(m_2)$ or
\item $|m_1|=|m_2|$, $\deg(m_1)=\deg(m_2)$ and $m_1$ is less than $m_2$ with respect to the left-to-right
lexicographical order when both $m_1$ and $m_2$ are written in {\em decreasing} order: $m_1=x_{i_1}\dots x_{i_k}$,
$m_2=x_{j_1}\dots x_{j_k}$ with $i_k\leq i_{k-1}\leq\dots\leq i_1$, $j_k\leq j_{k-1}\leq\dots\leq j_1$.
 (Equivalently, we may write $m_1$ and $m_2$ in increasing order and compare them with the right-to-left lexicographic order.)
\end{itemize}
 Note, that in  the ordering $<$ we compare the smallest letters first, and in the  ordering $\prec$ we compare the biggest letters first.
It is easy to see that both orderings are well-orderings on $[X]$ compatible with multiplication.

\begin{proof}[Proof of Lemma~\ref{lemmaP1}]
Our ordering $<$ satisfies the descending chain condition, and thus
Lemma~\ref{lemmaP1} can  easily be obtained by repeated application of the following sublemma:

 \medskip

\noindent {\bf Sublemma: }
There exist positive integers $k$ and $n$ such that every
monomial $m\in [X]$  of length $k$ with $m\geq x_n^k$ satisfies
$$
m=h+\sum c_s m_s,
$$
where $h\in I$ is degree-homogeneous with $\deg(h) = \deg(m)$, the sum is finite, $c_s\in\K$, and the $m_s$ are monomials of degree $\deg(m)$
such that $m_s<m$ for each $s$.

\medskip

To prove the sublemma, let $\overline{g}$ be the leading monomial of $g$ with respect to $\prec$, and let $ k = |\overline{g}| = |g|$.
 Without loss of generality
we can assume that $\overline{g}$ features in $g$ with coefficient $1$. We write $\overline g$
in an increasing way: $\overline{g}=x_{i_1}\dots x_{i_k}$ with $i_1\leq i_2\leq\dots\leq i_k$. Pick any positive integer $n$ such that $n\geq 2i_k+1$.
We shall show that these $n$ and $k$ satisfy the required conditions.

Let $m\in [X]$  of length $k$ be such that $m\geq x_n^k$.
Then $m=x_{j_1}\dots x_{j_k}$ with
$n\leq j_1\leq\dots\leq j_k$.
If $a \in \ZZ_{\geq 1}$, let
  $d_a:\K[X]\to \K[X]$ be the derivation defined by $d_a(u)=\{u, x_a\}$, extending via the Leibniz rule.
  Note that  $d_a$ is a graded derivation:  if applied to a degree-homogeneous polynomial  $f$, then   $d_a(f)$ is degree-homogeneous of degree $\deg (d_a(f) )= \deg(f)+a$.

  Consider
$$
h=d_{j_1-i_k}d_{j_2-i_{k-1}}\dots d_{j_k-i_1}(g).
$$
Since $g \in I$ and $I$ is  a Poisson ideal, $h\in I$.
Further, as the $d_a$ are graded and $g$ is degree-homogenous, so is $h$.
The proof will be complete if we verify that
$$
h=cm+\sum{c_s}m_s,
$$
where $c\neq0$, $c_s\in\K$, and the  $m_s$ are monomials such that $m_s<m$ for each $s$.

Let us apply the sequence of derivations
$d_{j_1-i_k}d_{j_2-i_{k-1}}\dots d_{j_k-i_1}$ to a monomial $u$ occurring in $g$.
By  the Leibniz rule
we get a sum of monomials with coefficients obtained by prescribing which of the derivations
acts on each letter of $u$.

Note  (assuming that $d_a(u) \neq 0$) that $d_a(u)$ has the same length as $u$.
Thus monomials in $h$ obtained from monomials in $g$ of length $<k$ are themselves of length $<k$ and therefore are smaller
than $m$ with respect to $<$.

Suppose now that $u$ has
length $k$.
Then there are two options: either
different differentials act on letters in different places in the monomial $u$ or this is not the case.
We call the
first of these ways {\it permutational} and the second {\it non-permutational}.
Monomials in $h$ obtained from $u$ in a non-permutational way will have at least one letter unchanged and therefore will have at least one letter
$x_i$ with $i\leq i_k<n$. Hence such monomials of $h$ are again  smaller than $m$ with respect to $<$.

It remains to
consider monomials of $h$ obtained from monomials  in $g$ of length $k$ in a
permutational way. For each  monomial  $u=x_{p_1}\dots x_{p_k}$ with $p_1\leq \dots\leq p_k$ occurring in $g$ and each permutation $\sigma\in \mf{S}_k$, we obtain a monomial $w$ of   $h$ given by
\[w=x_{j_1-i_k+p_{\sigma(k)}}\dots x_{j_k-i_1+p_{\sigma(1)}},\]
occurring with coefficient
\[\prod_{\ell=1}^k (j_\ell - i_{k-\ell+1}- p_{\sigma(k-\ell+1)}).\]
Since $u\preceq\overline{g}$, thus $p_{\sigma(k)}\leq p_k\leq i_k$. Hence $j_1-i_k+p_{\sigma(k)}\leq j_1$ and the equality
$j_1-i_k+p_{\sigma(k)}=j_1$ holds if and only if $p_{\sigma(k)}=p_k=i_k$. If $j_1-i_k+p_{\sigma(k)}<j_1$, then
$w<m$ and we are done. If $j_1-i_k+p_{\sigma(k)}=j_1$, we have $p_{\sigma(k)}=p_k=i_k$. Since $u\preceq\overline{g}$ and
$p_{\sigma(k)}=p_k=i_k$, we have $p_{\sigma(k-1)}\leq p_{k-1}\leq i_{k-1}$. Hence $j_2-i_{k-1}+p_{\sigma(k-1)}\leq j_2$ and the equality
$j_2-i_{k-1}+p_{\sigma(k-1)}=j_2$ holds if and only if $p_{\sigma(k-1)}=p_{k-1}=i_{k-1}$.
Repeating  the procedure,
we see that  $w\leq m$  and that $w=m$ only if $u=\overline{g}$ and the permutation $\sigma$ satisfies
$i_{\sigma(s)}=i_s$ for $1\leq s\leq k$.

Now, for each such $\sigma$, since $n\geq 2i_k+1$, we claim that the coefficient of the monomial $m$ is a positive integer.
Indeed, the coefficient is a product of  factors  of the form $j_\ell-i_{k-\ell+1}-p_{\sigma(k-\ell+1)}$, which  are positive  since $i_{k-\ell+1}+ p_{\sigma(k-\ell+1)}\leq 2i_k <n \leq j_\ell$.
Thus the coefficient with which $m$ occurs in  $h$ is nonzero.   The other monomials in $h$ are $<m$. This completes the proof of the sublemma and thus of the lemma.
\end{proof}

Note that we used in this proof only that $I$ is a module over $\K$ defined by the bracket multiplication (in other words a submodule of $S(W_+)$ under the adjoint action of $W_+$).

We now prove Theorem~\ref{theoremP}.  The key point of the proof is that, thanks to Lemma~\ref{lemmaP1}, $A = S(W_+)/I$ is spanned by the set
of monomials $m$ in $[X]$ which admit a factorisation $m=m_1m_2$, where
$m_1$ is a  monomial in $x_1,x_2,\dots,x_{n-1}$ and $|m_2|<k$.  Thus to estimate the growth of $A$ it suffices to count such monomials.

\begin{proof}[Proof of Theorem~\ref{theoremP}.]
As $S(W_+)$ is finitely graded by degree, it is standard (see \cite[Proposition~6.6]{KL}) that it suffices to show that $S(W_+)/I$ has polynomial growth if $I$ is a nonzero degree-graded Poisson ideal, and it clearly suffices to consider the case that $I $ is the Poisson ideal generated by a  single nonzero degree-homogenous element $g$.
Let $k$ and $n$ be the numbers produced by applying Lemma~\ref{lemmaP1} to $g$, and let
  $S$ be the set of all  monomials $m$ in $[X]$ which admit a factorisation $m=m_1m_2$, where
$m_1$ is a  monomial in $x_1,x_2,\dots,x_{n-1}$ and $|m_2|<k$.

By Lemma~\ref{lemmaP1},
each monomial $m\in[X]\setminus S$ can be  written, modulo $I$, as a linear combination of monomials of degree $\deg(m)$,
each of which either has length strictly less than $|m|$ or has length $|m|$ and features strictly fewer $x_i$ with
$i\geq n$. Applying this observation repeatedly, we see that every monomial $m\in[X]\setminus S$
can be written,  modulo $I$, as a linear combination of monomials from $S$ of the same degree as $m$.
We will call such presentation a {\it normal form} of $m$.
 Hence the image of $S$ in $A=\K[X]/I$
spans $A$, and for fixed $N \in \NN$, the number of monomials in $S$ of degree not exceeding $N$  provides an upper bound for $\dim \{u \in A \ |\  \deg u \leq N\}$.
As $A$ is finitely $\NN$-graded by degree, it  is standard that the growth of this dimension bounds $\GK A$.

It remains to estimate the number $p(N)$ of elements of $S$ of degree at most $N$. Clearly $p(N)\leq q(N)r(N)$, where
$q(N)$ is the number of monomials in $x_1,\dots,x_{n-1}$ of degree at most $N$ and $r(N)$ is the number of monomials  in $x_1,x_{2},\dots$ of degree
at most $N$ and length at most $k-1$.
First, $q(N)$ does not exceed the number
of monomials in $x_1,\dots,x_{n-1}$ of length at most $N$, which is $\left({N+n\atop N}\right)$. Thus there is a positive constant $c$ so that $q(N)\leq cN^n$ for all $N$.
On the other hand, in a degree at most $N$ monomial of length at most $k-1$ in $x_n,x_{n+1},\dots$ there are no more
than $N$ options for each letter and therefore  $r(N)\leq N^{k-1}$ for all $N$.
Hence $p(N)\leq cN^{k+n-1}$.

Hence $\GK(A)\leq k+n-1$, and $A$ has polynomial growth,
as required.
\end{proof}

We note that Theorem~\ref{theoremP} is also proved in \cite{PS} (see Corollary~2.13), with a much less constructive proof.

To end the section, we fix a positive integer $k$ and consider the symmetric power $S^k(W_+)$. Let $g \in S^k(W_+)$ be a nonzero degree-homogeneous element, and let $I$ be the $W_+$-submodule of $S^k(W_+)$ generated by $g$.  As noted after the proof of Lemma~\ref{lemmaP1}, the reduction formula in Lemma~\ref{lemmaP1} still applies to $S^k(W_+)/I$, and the argument in the proof of Theorem~\ref{theoremP} now gives that $\GK S^k(W_+)/I \leq k-1$.  As $S^k(W_+)$ clearly has GK-dimension $k$, we obtain:
\begin{proposition}\label{prop:Sk}
As a $W_+$-module, $S^k(W_+)$ is GK $k$-critical. \qed
\end{proposition}
For $k=2$, this was shown in \cite[Corollary~4.15]{PS}.

\section{The universal enveloping algebra of the positive Witt algebra}\label{U}

That $U(W_+)$ has just infinite GK-dimension follows from Theorem~\ref{theoremP} using the Poisson GK-dimension defined in \cite{PS} and \cite[Theorem~3.19]{PS}.
However, a direct proof, which we give here, is also straightforward; the techniques of Section~\ref{P}  apply also to $U(W_+)$.

We begin by giving a noncommutative version of  the reduction formula of Lemma~\ref{lemmaP1}.
Our result is more general than  needed here, for later use when considering quotients of $U(W)$.

By the Poincar\'e-Birkhoff-Witt theorem, $U(W)$ has a basis of monomials $e_{i_1}e_{i_2} \dots e_{i_k}$ with $i_1 \leq i_2 \leq \dots \leq i_k$.  We call such monomials {\em standard}.

\begin{lemma}\label{lem:newreduction}
Let $0 \neq g \in U(W)$, and let $I = U(W_+)gU(W_+)$ be the $U(W_+)$-sub-bimodule of $U(W)$ generated by $g$.
Then there exist positive integers $k$ and $n$ and an integer $\ell$ so that every standard monomial $m = e_{j_1}\dots e_{j_k}$ with $n \leq j_1 \leq \dots \leq j_k$ satisfies
\beq\label{NCRF} m = h + \sum c_t m_t,\eeq
where $h\in I$, the sum is finite, $c_t \in \KK$, and the $m_t$ are standard monomials so that for each $t$, we have $i \geq \ell$ for all letters $e_i$ featuring in $m_t$, and either
$|m_t| <k$ or $|m_t| = k$ and $i< n$ for at least one of the letters featuring in $m_t$.
Further, if $h$ is degree-homogeneous, then $ \deg (m) = \deg(h) = \deg(m_t)$ for all $t$.
\end{lemma}

\begin{proof}
Let $k = |g|$.
Writing $g$ as a sum of standard monomials, let $e_{\ell}$ be the smallest letter occurring in $g$ and define $n'$ so that $e_{n'}$ is the largest letter in $g$.  Let $n = 2|n'|+1$.

There are well-defined monomial orderings $<$ and $\prec$ on standard monomials in $U(W)$, defined just as the corresponding orderings on commutative monomials $x_{i_1}\dots x_{i_k}$ in the previous section.
Note that $<$ does not satisfy the descending chain condition because $W$ has no least element, but the induced order on standard monomials in letters $ \geq \ell$ does satisfy d.c.c.
Thus, as in  the proof of  Lemma~\ref{lemmaP1}, it suffices to show that we can rewrite $m$, modulo $I$, as a linear combination of standard monomials in letters $\geq \ell$, each of which are $< m$.

For any $a \in \ZZ$, let $\del_a = [\blank, e_a]$ as a linear operator from $U(W)\to U(W)$.
Recall that length defines a filtration on $U(W)$ whose associated graded ring is $S(W)$; for $f \in U(W)$ let $\gr(f)$ be the element of $S(W)$ associated to $f$, so $x_i = \gr(e_i)$.
For any $p \in U(W)$ and any $a \in \ZZ$, we have
\beq \label{greq}
\gr \del_a(p) = d_a(\gr( p)) \quad \quad \mbox{if $d_a( \gr (p)) \neq 0$,}
\eeq
where $d_a = \{\blank, x_a\}$ as in the proof of Lemma~\ref{lemmaP1}.

Let $\overline{g}$ be the $\prec$-leading standard monomial in $g$ and write $\overline{g} = e_{i_1} \dots e_{i_k}$ with $i_1 \leq \dots \leq i_k$.
Let
\[
h=\del_{j_1-i_k}\del_{j_2-i_{k-1}}\dots \del_{j_k-i_1}(g).
\]
 Since for any letter $x_{p}$ featuring in $\gr(g)$, and for any $d_a$ which is applied to $x_{p}$, by our choice of $n$ we have $a>p$, as in the proof of Lemma~\ref{lemmaP1}.
 Thus, just as in that proof,  and using \eqref{greq},
 \[  \gr(h) =d_{j_1-i_k}d_{j_2-i_{k-1}}\dots d_{j_k-i_1}(\gr(g)) = c \gr(m) + \sum c_t \gr(m'_t),\]
 where $c \neq 0$, the sum is finite, $c_t \in \KK$, and the $m'_t$ are standard monomials so that $m'_t < m$ for all $t$.
 Thus  the length of $h - m - \sum c_t m'_t$ is strictly smaller than $k$, so   $h-m$ is a linear combination of standard monomials which are all  strictly $<m$.

 Since $a>0$ for all $\del_a$ we have applied, only letters $\geq \ell$ occur in $h$.  Finally, as the $\del_a$ are graded linear maps, if $g$ is degree-homogenous so is $h$.
\end{proof}

\begin{theorem}\label{theoremW+}
Let $I$ be a nonzero two-sided ideal in $U(W_+)$. Then $A=U(W_+)/I$ has
polynomial growth.
\end{theorem}

\begin{proof} All essential points of the proof occur the proof of Theorem \ref{theoremP}.

 As before, since $U(W_+)$ is finitely graded by degree, we may assume that $I = (g)$ is the ideal generated by a single nonzero degree-homogenous element $g$.  Let $k = |g|$.

 Let $S$ be the set  of standard monomials  $m$ which admit a factorization $m=m_1m_2$, where
$m_1$ is a standard monomial in $e_1, \dots, e_{n-1}$ and  $|m_2| < k$.
It follows from the reduction formula in Lemma~\ref{lem:newreduction} that $U(W_+)/I$ is spanned by the image of $S$.  The same counting argument as in the proof of Theorem~\ref{theoremP} shows that $U(W_+)/I$ has polynomial growth.
\end{proof}

Theorem~\ref{theoremW+} gives the  first part of Theorem~\ref{ithm:U}, dealing with $U(W_+)$.

\section{The enveloping algebra of the full Witt algebra}\label{fw}

In this section, we consider the enveloping algebra of the full Witt algebra, and show that it has just infinite GK-dimension.
It clearly suffices to  show:
 \begin{theorem}\label{theoremW}
Let $I=(g)$ be a two-sided ideal in $U(W)$ generated by one nonzero element $g\in U(W)$. Then $A=U(W)/I$ has
polynomial growth.
\end{theorem}
Throughout the section, we fix the meanings of $g$, $I$, and $A$ as in the statement of Theorem~\ref{theoremW}.  Let $\pi:  U(W) \to A$ be the natural map.

Now,  $U(W)$ is finitely generated, say by  $\{ e_{-2}, e_{-1}, e_{1}, e_{2}  \}$, and thus so is $A$.
 Since  the growth of $A$ is controlled by the growth of any finite filtration, we are free to choose one that is convenient, but it will be a little bit more complicated this time to choose the right one.
  The problem is that unlike the situation for $W_+$, the usual degree function
  $\deg(e_i ) = i$
  does not give us a finite grading on $U(W)$;
  note that in the proofs of Theorems~\ref{theoremP} and \ref{theoremW+}, the finiteness of the degree grading played a crucial role.
Moreover, although of course there are many finite filtrations on $U(W)$, it is not necessarily clear how to choose one which  induces a filtration on the quotient with polynomial growth.

   Thus we will need to find an appropriate degree function which will give us a well-controlled finite filtration on $A $.
   We will see that a degree function of the form $\delta_C$, defined by
  $\delta_C(e_{i_1}\dots e_{i_k})=|i_1|+\dots+|i_k|+ C$, where $C$ is a constant, does the job.

 \subsection{A spanning set for $A$}

Our first step is to construct a set of standard monomials in $U(W)$ whose images span $A$.

Symmetrically to Lemma~\ref{lem:newreduction}, we have:

\begin{lemma}\label{fw2}
There exist positive integers $k$ and $n$ and an integer $\ell$ such that every
standard monomial $m=e_{j_1}\dots e_{j_k}$ with $ j_1\leq\dots\leq j_k \leq -n$  satisfies
$$
m=h+\sum c_t m_t,
$$
where $h\in I$, the sum is finite, $c_t \in \KK$, and the $m_t$ are standard monomials so that for each $t$, we have $i \leq \ell$ for all letters $e_i$ featuring in $m_t$, and either
$|m_t| <k$ or $|m_t| = k$ and $i>-n$ for at least one of the letters featuring in $m_t$.
\end{lemma}
\begin{proof}
Consider the automorphism $\Phi$ of $U(W)$ defined by $\Phi(x_i)=-x_{-i}$ for $i\in\Z$. If we apply Lemma~\ref{lem:newreduction} to the ideal $\Phi(I)$ and the monomial $\pm \Phi(m)$, we obtain that $\Phi(m)=h+\sum c_t m_t$,
where $h\in I$, the sum is finite, $c_t \in \KK$, and the $m_t$ are standard monomials so that for each $t$, we have $i \geq \ell$ for all letters $e_i$ featuring in $m_t$, and either
$|m_t| <k$ or $|m_t| = k$ and $i< n$ for at least one of the letters featuring in $m_t$.
Applying $\Phi$ to both sides once again, we arrive at the  result.
\end{proof}

Lemmata~\ref{lem:newreduction} and \ref{fw2} allow us to construct our spanning set, which we define here.

\begin{defn}\label{def:NS}
For positive integers $k, n$, let
 $NS(k,n)$ be the set of standard monomials  $m$ which admit a factorisation $m=aub$, where
$a$ is a standard monomial of length $< k$ in $e_{-n}$ and smaller letters,
$u$ is a standard  monomial in $e_{1-n},\dots e_0,e_1,e_2,\dots,e_{n-1}$, and
$b$ is a standard monomial of length $< k$ in $e_n$ and bigger letters.
\end{defn}

\begin{lemma}\label{lem:NS}
There exist positive integers  $k$ and $n$
such that  $A$ is spanned by the image of $NS(k,n)$.
\end{lemma}

\begin{proof}
As before, let   $k$ be the maximal length of monomials in $g$; that is $k = |g|$.

By Lemmata~\ref{lem:newreduction} and \ref{fw2}, there exist $n_1,n_2 \in \ZZ_{\geq 1} $ and $\ell_1,\ell_2 \in \Z $ such that for every standard monomial $m=e_{j_1}\dots e_{j_k} $, with $j_1 \leq \dots \leq j_k$, if  $j_1 \geq n_1$,  then
$$
m=\sum c_s m_s+h,
$$
where $h\in I$, the sum is finite, $c_s\in\K$, and
 the  $m_s$ are standard monomials such that $i \geq \ell_1$ for each $e_i$ occurring in $m_s$,  and  $i < n_1$ for some
$e_i$ occurring in $m_s$; and
if  $j_k \leq -n_2$,  then
$$
m=\sum c_s m_s+h,
$$
where  $h\in I$, the sum is finite, $c_s\in\K$, and
 the $m_s$ are standard monomials such that $i \leq \ell_2$  for each $e_i$ occurring in $m_s$, and $i > -n_2$ for some $e_i$ occurring  in $m_s$.

Let
$n={\rm} max\{ n_1,n_2, |\ell_1|,|\ell_2|\}$.
Repeatedly using the observations above to rewrite $m$ modulo $I$, we obtain the result.
\end{proof}

For the rest of the section, let $k, n$ be as given by Lemma~\ref{lem:NS}, and let $NS = NS(k,n)$.
We call the elements of $NS$ {\em normal words}, and a representation of $m \in A$ as a linear combination of (images of) normal words a {\em normal form} for $m$, bearing in mind that this normal form may not be unique.

The growth of $NS$ is polynomial, as we next show.

\begin{lemma}\label{countw}
For any  positive integer $C $, define a function
\[ \delta_C:  \{ \mbox{ standard monomials in $U(W)$ } \} \to \NN \]
by
 $\delta_C(e_{i_1}\dots e_{i_k})=|i_1|+{\dots}+|i_k|+C.$
 For any $N \in \NN$, let
 \[ p_C(N) = \# \{ w \in NS \ | \ \delta_C(w) \leq N \}.\]
 Then the function $p_C$ has polynomial growth.
\end{lemma}

\begin{proof}
The growth of $p_C$ does not depend on $C$, so without loss of generality let $C = 1$ and let $p = p_1$ and $\delta = \delta_1$.
For a standard monomial $m$, we refer to $\delta(m)$ as the {\em absolute degree} of $m$.
Note that $U(W)$ is {\em not } graded with respect to absolute degree.

 Clearly $p(N)\leq q(N)r(N)^2$, where
$q(N)$ is the number of standard monomials in
\[e_{1-n},\dots, e_0,e_1,\dots,e_{n-1}\]
 of absolute degree at most $N$,
 while $r(N)$ is the number of standard monomials in $e_n,e_{n+1},\dots$ of degree
at most $N$ and of length at most $k-1$.

Now, absolute degree, as a function on standard monomials, is  always greater than or equal to   length.
Thus $q(N)$ does not exceed the number
of standard monomials in $e_{1-n},\dots,e_{n-1}$ of length at most $N$, which is equal to  $\left({N+2n-1\atop N}\right)\leq cN^{2n-1}$,
for some  positive constant $c$ which depends only on $n$ and  not  on $N$.

On the other hand, in a monomial  in $e_n,e_{n+1},\dots$
of absolute degree at most $N$ and of length at most $k-1$, there are no more
than $N$ options for each letter and therefore  $r(N)\leq N^{k-1}$ for all $N$.
Hence $p(N) \leq c N^{2n+2k-3}$ and has polynomial growth.
\end{proof}

  \subsection{Choice of filtration\label{filt}}

It remains to estimate the growth of $A$ from  the spanning set constructed  in the previous subsection.

Since $A$ is finitely generated, the growth of  any finite filtration bounds the growth of $A$.
The main result of this subsection is that there is a constant $C$ such that the function $\delta_C$ induces a finite filtration of $A$, which by Lemma~\ref{countw} will have polynomial growth.

To have a filtration  $A_1 \subseteq A_2 \subseteq A_3\subseteq{\dots}$ on $A=\bigcup A_i$
 means to choose a map $\rho: A\to\N$, satisfying
\begin{equation}\label{star}
\text{$\rho(uv)\leq \rho(u)+\rho(v)$, $\rho(u+v)\leq \max\{\rho(u),\rho(v)\}$ and $\rho(\alpha u)=\rho(u)$,}
\end{equation}
for any $u,v \in A$ and  $\alpha \in \K^*$.
Suppose that for some $C$, the map $\delta_C: NS\to \N$ has the property that for any two normal words $m_1, m_2 \in  NS$,
we can find a normal form
\[ \pi(m_1 m_2) = \sum c_i \pi(w_i),\]
where the $c_i \in \KK$ and the $w_i$ are normal words so that
\begin{equation}\label{2star}
\text{$\delta_C(w_i)\leq \delta_C (m_1)+ \delta_C (m_2)$ for all $i$.}
\end{equation}
We claim that this is enough to construct the required $\rho$.
For, define
$$
\rho(u)=\min\limits_{\substack{\text{normal forms}\\ u =\sum c_j\pi(w_j)}}\Bigl\{\max\limits_j\delta_C(w_j)
\Bigr\}.
$$
Then $\rho$ is easily seen to  satisfy the required conditions (\ref{star}).

So our goal is to show that there is some constant $C$ so that $ \delta_C$ satisfies \eqref{2star}.

\begin{proposition}\label{prop:delta}
Let  $\ell =  \max\{ |i|: e_i\ \text{ \rm features in}\  g \}$.
Then \eqref{2star} holds for $C = 4k^2 \ell$, where we recall that $k = |g|$.
\end{proposition}

\begin{proof}
It suffices to show that for any normal words $m_1, m_2 \in NS$, and for any normal form $\pi(m_1 m_2) = \sum c_i \pi(w_i)$ for $\pi(m_1 m_2)$, where $w_i \in NS$ and $c_i \in \KK^*$, that
\beq\label{3star}
\delta_0(w_i) \leq \delta_0(m_1) + \delta_0(m_2) + C \quad \mbox{ for all $i$.}
\eeq
So we need to understand how $\delta_0$ behaves on the words appearing in a normal form for $\pi(m_1 m_2)$.

Recall the definition of $n$ from Lemma~\ref{lem:NS}.
Write
$$
\text{$m_1=a_1u_1b_1$ and $m_2=a_2u_2b_2$},
$$
where $u_1$ and $u_2$ are standard monomials of any length on variables with indices strictly between $-n$ and $n$, $a_1$ and $a_2$ are standard monomials of length $<k$ on letters with indices $\leq -n$, while $b_1$, $b_2$ are standard monomials of length $<k$ on letters with indices $\geq n$.
Now, normal words are standard monomials, and we first use the commutation relations $e_i e_j = e_j e_i + (j-i) e_{i+j}$ to write $m_1m_2$ as a linear combination of standard monomials, that is as a linear combination of words of the form
$$
m_3=a_3u_3b_3,
$$
where  $a_3$ is a standard monomial on letters with indices $\leq -n$, $u_3$ is a standard monomial  on letters with indices strictly between $-n$ and $n$, and $b_3$ is a standard monomial on letters with indices $\geq n$. Let us call those letters
with indexes $|i|\geq n$, {\it big letters}.
Note that in the course of applying the commutation relations, the total number of big letters does not increase.  Thus $a_3$ and $b_3$ have length $\leq 2k-2$, while $u_3$ may have any length.

Now we will use the reduction procedure from Lemma~\ref{lem:newreduction} to see how $\delta_0$ changes  when we get rid of  big letters in $b_3$. If the length of $b_3$ is less than $k$ we do not have to do anything.
Otherwise let $m$ be the monomial composed of the last $k$ letters of $b_3$.
According to Lemma~\ref{lem:newreduction},  to  get rid of one of the (big) letters in $m$, we find a sequence of derivations $D=\d_{a_1}\dots \d_{a_k}$, with all the $a_i \geq 1$, and some $c \in \KK^*$ such that
\beq\label{foo}
cD(g) = m+\sum c_s m'_s.
\eeq

Let $m'$ be some $m'_s$.  Now $m'$ is a standard monomial which falls into one of three cases:

\begin{itemize}
\item[I.]
$|m'| <|m|$, and $m'$ is obtained from some monomial $\tilde{g}$ in $g$ by applying $D$ and then the commutation relations;

\item[II.] $|m'| = |m|$ and $m'$ is obtained from some monomial $\tilde{g}$ in $g$ (which necessarily has length $k$) by a non-permutational action of $D$;

\item[III.]  $|m'| = |m|$  and $m'$ is obtained from some monomial $\tilde{g}$  in $g$ by a permutational action of $D$.
\end{itemize}

First, we note what happens to the number of big letters in each case.
 In case I,
since $m'$ is shorter than $m$, and all letters in $m$ are big, $m'$ contains fewer big letters than $m$.
In case
II, since the action is non-permutational, there is a letter in $m'$ which was present in $\tilde{g}$. But monomials of $g$ consist of letters which are not big, by definition of $n$. Thus the number of big letters in $m'$ is smaller than in $m$.
In case III, it may be that the number of big letters in $m'$ is still equal to $k$, which is the number of (big) letters in $m$, but certainly there are no more than $k = |m'|$ big letters in $m'$.
Further, by our choice of $n$, in case III all $e_i$ occurring in $m'$ have $i \geq 1$.

We now consider how $\delta_0$ changes throughout this process.  In situation III, as $m' \leq m$ and  $m'$ and $m$ are made of  letters $\geq e_1$, we have $\delta_0(m') = \deg(m') \leq \deg(m) = \delta_0(m)$.
In  cases I and II, $\delta_0(m')$ may be bigger than $\delta_0(m)$.
Since applying the commutation relations does not increase $\delta_0$, we may assume that the monomial $m'$ is a (possibly non standard) monomial obtained from applying $D$ to a monomial $\tilde{g}$ of $g$.

 Recall that $\bar{g}$ is the $\prec$-leading monomial of $g$, and $m$ is obtained from the monomial $\bar{g}$  by applying $D$.
 Set $\deg (D) = a_1+\dots + a_k$, so applying $D$ increases the degrees of homogenous elements by $\deg (D)$.
As $ \delta_0(e_{a+b}) \leq \delta_0(e_a) + \delta_0(e_b)$ and all the $a_i$ are $\geq 1$, we have
\[
\delta_0(m') \leq \delta_0(\tilde{g}) + \deg D.
\]
So
\begin{align*}
\delta_0(m')-\delta_0(m) & \leq \delta_0(\tilde{g}) + \deg D - \delta_0(m) \\
 & = \delta_0(\tilde{g}) + \deg D - \deg m \quad \quad  \mbox{as letters in $m$ are big} \\
 & \leq \delta_0(\tilde{g}) + |\deg m - \deg D | \\
 & = \delta_0(\tilde{g}) + | \deg \bar{g} | \\
 & \leq 2k \ell,
\end{align*}
by choice of $\ell$, as $\tilde{g}$ and $\bar{g}$ have no more than $k$ letters.

To summarize the discussion above:  in cases I and II, we remove at least one big letter and increase $\delta_0$ by no more than $2k \ell$.  In case III, we do not remove big letters and do not increase $\delta_0$.
As proved in Lemma~\ref{lem:newreduction}, after repeating this procedure finitely many times, we may write $m_1m_2$, modulo $I$,  as a linear combination of normal words.
However many times we repeat the procedure we remove a maximum of $k-1$ big letters from $b_3$, and thus we add a maximum of $2k(k-1) \ell$ to $\delta_0$.  Note that in applying this process to $b_3$, we never add any letters with indices $< -n$.

After we apply the procedure from Lemma~\ref{fw2} to $a_3$ at the other end of the word, we have found a normal form for $m_1m_2$ and have added maximum of  another $2k(k-1)\ell$ to $\delta_0$.
In other words, we have written
\[ \pi(m_1m_2) = \sum c_i \pi(w_i)\]
where the $w_i$ satisfy \eqref{3star}, as required.
\end{proof}

The  proof of  Theorem~\ref{theoremW} is now an easy combination of other results in this section.

\begin{proof}[Proof of Theorem~\ref{theoremW}]
Let $C$ be the constant given by Proposition~\ref{prop:delta}.
The discussion before that proposition shows that setting
\[ A(N) = \operatorname{span} \{ \pi(w) : w \in NS, \pi_C(w) \leq N \}\]
defines a finite filtration on $A$.
By Lemma~\ref{countw}, this filtration has polynomial growth, and so $\GK(A) < \infty$.
\end{proof}

Theorem~\ref{theoremW} gives the second part of Theorem~\ref{ithm:U}, dealing with $U(W)$.

\begin{remark}\label{rem:Wminus}
Let $\mathbf{W}_1$ be the first {\em Cartan algebra}, which is isomorphic to the subalgebra of $W$ spanned by $\{e_n : n \geq -1\}$.
This  is a simple graded Lie algebra of polynomial growth.
Similar methods to those used in Theorem~\ref{theoremW} show that $U(\mathbf{W}_1)$ has just infinite GK-dimension. To show this for all simple graded Lie algebras of polynomial growth is the subject of ongoing work.
\end{remark}

\section{Central quotients of the enveloping and symmetric algebras of $\V$}\label{V}
In this section we first prove that all the central quotients $U(\V)/(c-\kappa)$ have just infinite GK-dimension, completing the proof of Theorem~\ref{ithm:U}, and then consider the related Poisson algebras $S(\V)/(c-\kappa)$.
Because the  ideas of the proofs are similar to those in previous sections, we leave some details to the reader.

\subsection{Central quotients of the enveloping algebra of $\V$}
Fix $\kappa \in \KK$ and let $R=U(\V)/(c-\kappa)$.
Essentially the same argument as for the full Witt algebra works to show that  $R$ has just infinite GK-dimension; we give a sketch of the proof.

Note that $R$, like $U(W)$,  has a basis of standard monomials in the $e_i$.

\begin{defn}\label{def:NSprime}
For positive integers $k, n$, let
 $NS(k,n) \subset R$ be the set of standard monomials  $m$ in the $e_i$ which admit a factorisation $m=aub$, where
$a$ is a standard monomial of length $< k$ in $e_{-n}$ and smaller letters,
$u$ is a standard  monomial in $e_{1-n},\dots e_0,e_1,e_2,\dots,e_{n-1}$, and
$b$ is a standard monomial of length $< k$ in $e_n$ and bigger letters.
\end{defn}

\begin{lemma}\label{lem:NSkappa}
Let $g$ be a nonzero element of $R$, let $I = (g)$, and let $A = R/I$.
There exist positive integers  $k$ and $n$
such that  $A$ is spanned by the image of $NS(k,n)$.
\end{lemma}

\begin{proof}
The key point is that the reduction formulae in Lemmata~\ref{lem:newreduction} and \ref{fw2} still hold.
As before, let $\del_a = [\blank, e_a]$ as a linear operator on $R$ and consider the effect of applying some $\del_{a_1} \cdots \del_{a_k} $ to a standard monomial of length $k$.  If $\kappa \neq 0$ and some expression of the form $[e_{-a_j}, e_{a_j}]$ has been computed, we may obtain some standard monomials of length $< k$ in the result; but the leading term will have length $k$ and will be given by the procedures in the previous section.  Thus the proof of Lemma~\ref{lem:NS} goes through in this situation, almost without change.
\end{proof}

 \begin{theorem}\label{theoremWkappa}
Let $I=(g)$ be a two-sided ideal in $R$ generated by one nonzero element $g\in R$. Then $A=R/I$ has
polynomial growth.
\end{theorem}
\begin{proof}
We may define the functions $\delta_C$ just as with $U(W)$.  The only part of the argument which is different for $R$ is the proof of Proposition~\ref{prop:delta}.
When we compute
\beq\label{RFagain}
 c \del_{a_1} \dots \del_{a_k} (g) = m + \sum_s c_s m'_s,
\eeq
as in \eqref{foo},
consider some $m' = m'_s$ as before.
In addition to cases I, II, III as in the proof of Proposition~\ref{prop:delta}, we may have
\begin{itemize}
\item[I'.] $|m'|< |m|$, and $m'$ is obtained from some monomial $\tilde{g}$ in $g$ by applying  $\del_{a_1} \dots \del_{a_k}$ and then applying the relation
\[ [ e_{-i}, e_i ] = 2ie_{0} + \frac{i-i^3}{12} \delta_{i+j,0} \kappa.\]
\end{itemize}
As before, applying this relation does not increase $\delta_0$, so we may assume that $m'$ is a (possibly non standard) monomial obtained from applying $D$ to a monomial $\tilde{g}$ of $g$.
 The argument of Proposition~\ref{prop:delta} goes through with only minor changes, and the result follows just as in the proof of Theorem~\ref{theoremW}.
\end{proof}

Theorem~\ref{theoremWkappa} completes the proof of Theorem~\ref{ithm:U} by establishing the part dealing with $U(\V)$.

\begin{remark}\label{rem:locVir}
By the same method as in the proof of Theorem~\ref{theoremWkappa}, one may show that the localised enveloping algebra
$U(\V)  \otimes_{\K[c]}\K(c) $, considered as an algebra over $\K(c)$, has just infinite GK-dimension.
It follows, using a similar argument to the proof of \cite[Lemma~3.10]{KL}, that $U(\V)  \otimes_{\K[c]}\K(c) $ has just infinite GK-dimension considered as a $\K$-algebra.  We leave the details to the reader.
\end{remark}

\subsection{Central quotients of the symmetric algebra of $\V$}

In this subsection, let $\kappa \in \KK$ and let $R = S(W)/(c-\kappa)$.
Since $c-\kappa$ is Poisson central, $R$ is a Poisson algebra; in fact if we filter $U(\V)$ by setting $|e_i| =1$ and $|c| = 0$, then
$R$ is the associated graded ring of $U(\V)/(c-\kappa)$.
Note that if we define $d_a = \{ \blank, x_a\}$ and $\del_a = [\blank, e_a]$ as before, then \eqref{greq} still holds.

Similar arguments to those that have gone before prove:

\begin{theorem}\label{theoremPkappa}
Let $g$ be a nonzero element of $R$ and let $I = \{(g)\}$ be the Poisson ideal generated by $g$.
Then $A = R/I$ has polynomial growth.
\end{theorem}

\begin{proof}
The reduction process works as before:
writing $g = \gr(g')$ and computing
\[ h = c \del_{a_1}\dots \del_{a_k}(g') = m + \sum_s c_s m'_s\]
as in \eqref{RFagain}, by \eqref{greq}
\[ \gr(h) = \gr(m) + \sum_t c_t \gr(m'_t),\]
where the only $m'_t$ surviving have length $k$.
Thus as in Lemma~\ref{lem:NSkappa} there are $k$ and $n$ so that $A$ is spanned by the image of $\gr(NS(k,n))$.

Comparing $\delta_0(\gr(m'))$ with $\delta_0(\gr(m))$ as in the proof of Proposition~\ref{prop:delta} we see that only cases II and III occur.
The conclusion of Proposition~\ref{prop:delta} still holds, and so as in the proof of Theorem~\ref{theoremW} $\GK(A) < \infty$.
\end{proof}

\begin{remark}\label{rem:SWminus}
Similarly, one may show that $S(\mathbf{W}_1)$ has just infinite GK-dimension.
We omit the proof.
\end{remark}

\section{Applications}

In this section we give several applications of Theorem~\ref{ithm:U}.  We first give a short proof that Verma modules for $\V$ are faithful over the appropriate central factor of $U(\V)$.  (A more direct proof is an unpublished result of Nolan Wallach \cite{Wallach}.)  We next prove that $U(W_+)$, $U(W)$, and $U(\V)$ all satisfy the ascending chain condition on completely prime ideals.
As a consequence, these algebras are {\em Hopfian}:  they are not isomorphic to any proper quotient.

\subsection{Annihilators of induced modules}

Fix $\lambda, \kappa \in \CC$.  Note that the Virasoro algebra $\V$ has a triangular decomposition:  define $\mf n_+ := \CC(e_n:  n \geq 1)$, $\mf h := \CC(c, e_0)$, and $\mf n_- := \CC(e_n : n \leq -1)$.  Let   $\mf b_+ := \mf n_+ \oplus \mf h$.
Let $\CC_{\kappa, \lambda}$ be the one-dimensional representation of $\mf b_+$ where $\mf n_+$ acts trivially, $c$ acts as $\kappa$, and $e_0$ acts as $\lambda$.
Then define the {\em Verma module} $M_{\kappa, \lambda}$ to be $\U(\V) \otimes_{\U(\mf b_+)} \CC_{\kappa, \lambda}$.
It is immediate that
\[M_{\kappa, \lambda} \cong \U(\V)/\U(\V) ( c-\kappa, e_0 -\lambda, e_n: n \geq 1)\]
 and that $M_{\kappa, \lambda}$ is non-positively graded, with $\dim (M_{\kappa, \lambda})_{-n} = \mc P(n)$, the $n$'th partition number.

Verma modules are  examples of the larger class we call, slightly imprecisely, {\em induced modules}.  These are modules of the form $M = U(\V) \otimes_{U(\mf b_+)} M'$, where $M'$ is a
  representation of $\mf{b}_+$.
Besides Verma modules, examples include {\em logarithmic representations}, where $\dim_\CC M' < \infty$ and where $\mf n_+$ acts trivially on $M'$, $c$ acts as a scalar, and $e_0$ acts as a non-semisimple matrix.  These representations are important in logarithmic conformal field theory, see \cite{GK}.

 {\em Whittaker modules} \cite{OW} form another class of examples.
Here let $\mf n' = \mf n \oplus \CC c$ and let $M''$ be the one-dimensional $\mf n'$ module  where $c$ acts as a scalar, and the $e_n$ act trivially for $n \geq 3$.  The module
\[ M = U(\V) \otimes_{U(\mf n')} M'' \cong U(\V)\otimes_{U(\mf b_+)} U(\mf b_+) \otimes_{U(\mf n')} M''\]
is a {\em Whittaker module}.
If $e_1, e_2$ act nontrivially on $M''$, then $M$ is simple by \cite[Corollary~4.5]{OW}.
All of these examples are annihilated by some $c- \kappa$, where $\kappa \in \CC$, and so have  {\em central character} $\kappa$.

Using
 Theorem~\ref{ithm:U}, we may immediately compute the annihilator of an induced module.

 \begin{theorem}\label{thm:induced}
 Let $M'$ be a representation of $\mf b_+$ with central character $\kappa \in \CC$.
 Let $M = U(\V) \otimes_{U(\mf b_+)} M'$.
 Then $\Ann_{\U(\V)} M = (c-\kappa)$.
 \end{theorem}
 \begin{proof}
 Clearly $(c-\kappa) M = 0$.

 Let $\mc P$ be the set of {\em negative partitions} $\underline{\lambda} = (\lambda_1, \dots, \lambda_k) $ where the $\lambda_i $ are negative integers with $\lambda_1 \leq \lambda_2 \leq \dots \leq \lambda_k $.
 If $\underline{\lambda} = (\lambda_1, \dots, \lambda_k) \in \mc P$, let $e_{\underline \lambda} = e_{\lambda_1} \dots e_{\lambda_k}$.
 If $0 \neq m \in M'$, it follows from the Poincar\'e-Birkhoff-Witt theorem that the elements $\{ e_{\underline{\lambda}} \otimes m: \underline{\lambda}\in \mc P\}$ are linearly independent in $M$, and thus, as these elements are in bijection with partitions, that $M$ has subexponential growth and infinite Gelfand-Kirillov dimension.

 Let $K = \Ann_{\U(\V)} M$.  If $K \supsetneqq (c-\kappa)$ then by Theorem~\ref{ithm:U} $U(\V)/K$ has polynomial growth and thus, by \cite[Proposition~5.1(d)]{KL}, so does $M$.
 This contradiction shows that $K = (c-\kappa)$.
 \end{proof}

 \begin{corollary}\label{cor:Verma}
%\begin{proposition}\label{prop:Verma}
For any $\kappa, \lambda \in\CC$, the Verma module $M_{\kappa, \lambda}$ is a faithful $\U(\V)/(c - \kappa)$-module. \qed
%\end{proposition}
\end{corollary}

%\begin{proof}
%Let $K = \Ann_{\U(\V)} M_{\kappa, \lambda}$.
%Clearly $c - \kappa \in K$.
%If $K \supsetneqq (c- \kappa)$ then by Theorem~\ref{theoremWkappa}, $ \U(\V)/K$ has polynomial growth.
%However, $M_{\kappa, \lambda}$ is finitely generated over  $ \U(V)/K$.
%Because the dimension of $(M_{\kappa, \lambda})_{-n}$ grows subexponentially with $n$, the GK-dimension of $ M_{\kappa, \lambda} $ is infinite.   This is a contradiction.
%\end{proof}

 Corollary~\ref{cor:Verma} is an unpublished result of Nolan Wallach \cite{Wallach}, and is independently due to Olivier Mathieu in unpublished work; see \cite[footnote 2, p. 496]{CM}.  We thank Rupert Wei Tze Yu for pointing out this reference to us.

It is known \cite[Theorem~1.2]{FF} that for any $\kappa$, the module $M_{\kappa, \lambda}$ is simple for generic $\lambda$.  Thus it follows immediately that $\U(\V)/(c-\kappa)$ is primitive.

%A {\em logarithmic} representation of $\V$ is induced from a finite-dimensional representation of $\mf b_+$ where $\mf n_+$ acts trivially, $c$ acts as a scalar, and $e_0$ acts as a non-semisimple matrix.  These representations are important in logarithmic conformal field theory, see \cite{GK}.  The proof above also gives:
%\begin{proposition}\label{prop:logarithmic}
\begin{corollary}\label{cor:log or Whittaker}
Let $N$ be a logarithmic representation or a Whittaker module over  $\V$.  Then $\Ann_{\U(\V)}(N) = (c - \kappa)$ for some $\kappa \in \CC$. \qed
\end{corollary}
%\end{proposition}

\subsection{Completely prime ideals}

In \cite[Conjecture~1.3]{PS}, it is conjectured that $U(W_+)$ satisfies the ascending chain condition on two-sided ideals.
We cannot prove this, but we do show

\begin{proposition}\label{prop:ACC}
The algebras
$U(W_+)$ and $U(\V)$ satisfy the ascending chain condition (ACC) on completely prime ideals.
\end{proposition}

\begin{proof}
We first note that any ring $R$ of finite or just infinite GK-dimension satisfies ACC on completely prime ideals.
Letting $P_0$ be the first ideal in the chain, it is sufficient to show that $R/P_0$ has ACC on completely prime ideals.
Thus we may replace $R$ by $R/P_0$ and assume that  $R$ is a domain with $\GK R <\infty$.
Now if $I$ is a nonzero ideal of $R$, then by \cite[Proposition~3.15]{KL}, $\GK R/I \leq \GK R -1$,
so by induction the length of a chain of completely prime ideals is bounded by $\GK R$.
Thus by Theorem~\ref{ithm:U}, $U(W_+)$ and $U(\V)/(c-\kappa)$ (for any $\kappa \in \K$) have ACC on completely prime ideals.

In fact, note that if $f[x] \in\KK[x]$ is irreducible, then $U(\V)/(f(c))$ has just infinite GK-dimension and thus ACC on completely prime ideals.
To see this, let $\KK'$ be the extension field $\KK[x]/f(x)$ of $\KK$, and note that
\[U(\V)/(f(c) ) = U_{\KK}(\V)/(f(c)) \cong U_{\KK'}(\V)/(c-x).\]
This last has just infinite GK-dimension by Theorem~\ref{ithm:U}.

We now consider an ascending chain $P_1 \subseteq P_2 \subseteq \cdots$ of completely prime ideals of $U(\V)$.
If $\bigcup P_n$ contains a nonzero element of $\K[c]$, then as the $P_n$ are prime and $c$ is central, some $P_n$ contains an irreducible polynomial $f(c) \in \KK[c]$.
By the first part of the proof, therefore, the chain stabilizes.

So we may assume that each $P_n \cap \K[c] = 0$.
As $P_n$ is prime, each $U(\V)/P_n$ is $\K[c]$-torsionfree.
Thus if $P_n \neq P_{n+1}$, then $(P_{n+1}/P_n) \otimes_{\K[c]}\K(c) \neq 0$ and so
\[ P_n \otimes_{\K[c]}\K(c) \neq P_{n+1}  \otimes_{\K[c]}\K(c) .\]
Further, these ideals are completely prime as 
\[U(\V)  \otimes_{\K[c]}\K(c)/ P_n  \otimes_{\K[c]}\K(c) \cong (U(\V)/P_n)  \otimes_{\K[c]}\K(c)\]
 is a domain.

Thus it suffices to show that $U(\V)  \otimes_{\K[c]}\K(c) $ has ACC on completely prime ideals.
By Remark~\ref{rem:locVir}, $U(\V)  \otimes_{\K[c]}\K(c) $ has just infinite GK-dimension.
Thus by the first part of the proof, $U(\V)  \otimes_{\K[c]}\K(c) $ satisfies the ACC on completely prime ideals.
\end{proof}

\subsection{The Hopfian and Bassian properties}

To end the paper, we consider two ring-theoretic properties which are related to noetherianity.  A ring $R$ is {\em Hopfian} if $R$ is not isomorphic to any proper quotient $R/J$ (equivalently, any epimorphism from $R \to R$ is an isomorphism).
More strongly, $R$ is {\em Bassian} if there is no injection of $R$ into any proper quotient $R/J$.
We thank Lance Small for introducing us to these concepts.

\begin{proposition}\label{prop:HB}
The algebras $U(W_+)$, $U(W)$, $U(\mathbf{W}_1)$, and $U(\V)/(c -\kappa)$ are Bassian and Hopfian, and $U(\V)$ is Hopfian.
\end{proposition}

That $U(W_+)$ is Hopfian is proved in \cite[Remarks~2.2]{RS}, and \cite[Section 4]{RS} asks whether $U(W)$ is Bassian or Hopfian.

\begin{proof}
If $R$ has just infinite GK-dimension, then $\GK R/J < \GK R$ for any proper ideal $J$ of $R$, so $R$ cannot inject in $R/J$.  Thus the Bassian (and thus Hopfian) property for $U(W_+)$, $U(W)$, and $U(\V)/(c -\kappa)$ follows from Theorem~\ref{ithm:U}.
For $U(\mathbf{W}_1)$ it follows from Remark~\ref{rem:Wminus}.

To show that $U(\V)$ is Hopfian, let   $R = U(\V)$ and let $f$ be a surjective endomorphism of $R$, with kernel $J$.   As $R/J \cong \operatorname{Im}(f)$ is torsionfree as a module over $\KK[c]$,
the complex
\[0 \to J \otimes_{\KK[c]} \KK(c)  \to R \otimes_{\KK[c]} \KK(c) \to ( R/J) \otimes_{\KK[c]} \KK(c) \to 0\]
is exact.  Now by Remark~\ref{rem:locVir}, we must have $J \otimes_{\KK[c]} \KK(c) = 0$.
As $R$ is $\KK[c]$-torsionfree, $J = 0$.
\end{proof}

That $U(\V)$ is Hopfian also follows from \cite[Corollary~2.6]{RS} and Proposition~\ref{prop:ACC}.
We do not know whether $U(\V)$ is Bassian.

\medskip

{\bf Acknowledgements:} \
This work is funded by the EPSRC grant EP/M008460/1/.
The first named author is grateful to IHES and MPIM,  where   part of this work was done,
for hospitality and support.

We thank Jos\'e Figueroa-O'Farrill, Tom Lenagan,  and Lance Small for useful comments and discussions.

 \bibliographystyle{amsalpha}
 \bibliography{Wittbiblio}

\end{document}